\UseRawInputEncoding
\documentclass[11pt]{amsart}
\usepackage{latexsym,amssymb,amscd}
\usepackage{epsfig}
\usepackage{graphics}
\usepackage{float}
\usepackage{epstopdf}
\usepackage{graphicx}
\usepackage{amsmath}
\def\B'c{{\mathcal{B'}}}
\def\U'c{{\mathcal{U'}}}

\def\opn#1#2{\def#1{\operatorname{#2}}} 
\opn\chara{char} \opn\length{\ell}
\opn\projdim{proj\,dim} \opn\injdim{inj\,dim} \opn\ini{in}
\opn\rank{rank} \opn\depth{depth} \opn\sdepth{sdepth}
\opn\height{ht} \opn\embdim{emb\,dim} \opn\codim{codim}

\opn\Tr{Tr} \opn\bigrank{big\,rank}
\opn\superheight{superheight}\opn\lcm{lcm}
\opn\trdeg{tr\,deg}%
\opn\reg{reg} \opn\lreg{lreg} \opn\set{set} \opn\supp{Supp}
\opn\shad{Shad}
\opn\div{div} \opn\Div{Div} \opn\cl{cl} \opn\Cl{Cl}
\opn\Spec{Spec} \opn\Supp{Supp} \opn\supp{supp} \opn\Sing{Sing}
\opn\Ass{Ass} \opn\Min{Min} \opn\size{size} \opn\bigsize{bigsize}
\opn\lex{lex}
\opn\Ann{Ann} \opn\Rad{Rad} \opn\Soc{Soc}
\opn\Ker{Ker} \opn\Coker{Coker} \opn\Im{Im} \opn\Hom{Hom}
\opn\Tor{Tor} \opn\Ext{Ext} \opn\End{End} \opn\Aut{Aut} \opn\id{id}

\opn\nat{nat} \opn\GL{GL} \opn\SL{SL} \opn\mod{mod} \opn\ord{ord}
\opn\aff{aff} \opn\con{conv} \opn\relint{relint} \opn\st{st}
\opn\lk{lk} \opn\cn{cn} \opn\core{core} \opn\vol{vol}
\opn\gr{gr}

\def\pot#1#2{#1[\kern-0.28ex[#2]\kern-0.28ex]}

\opn\dirlim{\underrightarrow{\lim}}
\opn\invlim{\underleftarrow{\lim}}
\def\pnt{{\raise0.5mm\hbox{\large\bf.}}}

\def\Implies{\ifmmode\Longrightarrow \else
     \unskip${}\Longrightarrow{}$\ignorespaces\fi}
\def\implies{\ifmmode\Rightarrow \else
     \unskip${}\Rightarrow{}$\ignorespaces\fi}
\def\iff{\ifmmode\Longleftrightarrow \else
     \unskip${}\Longleftrightarrow{}$\ignorespaces\fi}

\let\:=\colon
\newtheorem{Theorem}{Theorem}[section]
\newtheorem{Lemma}[Theorem]{Lemma}
\newtheorem{Corollary}[Theorem]{Corollary}
\newtheorem{Proposition}[Theorem]{Proposition}
\newtheorem{Remark}[Theorem]{Remark}

\newtheorem{Example}[Theorem]{Example}

\newtheorem{Definition}[Theorem]{Definition}

\let\epsilon=\varepsilon
\let\phi=\varphi
\let\kappa=\varkappa
\numberwithin{equation}{section}

\title{Depth and Stanley depth of powers of the edge ideals of some caterpillar and lobster trees}
\author[Tooba Zahid] {Tooba Zahid}
\address{Tooba Zahid, School of Natural Sciences, National University of Sciences and Technology Islamabad, Sector H-12, Islamabad Pakistan.}
\email{tooba14325@gmail.com}

\author[Zunaira Sajid] {Zunaira Sajid}
\address{Zunaira Sajid, School of Natural Sciences, National University of Sciences and Technology Islamabad, Sector H-12, Islamabad Pakistan.}
\email{zunaira$\_\,$sajjid@yahoo.com}

\author[Muhammad Ishaq]{Muhammad Ishaq}
\address{Muhammad Ishaq, School of Natural Sciences, National University of Sciences and Technology Islamabad, Sector H-12, Islamabad Pakistan.}
\email{ishaq$\_\,$maths@yahoo.com}
\begin{document}
	\maketitle
	\begin{abstract}
		Let $S$ be a ring of polynomials in finitely many variables over a field. In this paper we give lower bounds for depth and Stanley depth of modules of the type $S/I^t$ for $t\geq1$, where $I$ is the edge ideal of some caterpillar and lobster trees. These new bounds are much sharper than the existing bounds for the classes of ideals we considered.  
	\end{abstract}
	\section*{Introduction}
Let $K$ be a field and $S=K[x_1,\dots,x_m]$ be the polynomial ring in $m$ variables over $K$. Let $N$ be a finitely generated $\mathbb{Z}^m$-graded $S$-module. Let $uK[Z]$ be the $K$-subspace generated by all elements of the form $uy$ where $u$ is a homogeneous element in $N$, $y$ is a monomial in $K[Z]$ and $Z\subseteq\{x_1,x_2,\dots,x_m\}$. If $uK[Z]$ is a free $K[Z]$-module then it is called a Stanley space of dimension $|Z|$. A decomposition $\mathcal{D}$ of the $K$-vector space $N$ as a finite direct sum of Stanley spaces is called a Stanley decomposition of $N$. Let $$\mathcal{D} \, : \, N = \bigoplus_{j=1}^r u_j K[Z_j].$$ The Stanley depth of $\mathcal{D}$ is $\sdepth(\mathcal{D})=\min\{|Z_j|\}$. The number
$$\sdepth(N):=\max \{\sdepth(\mathcal{D})\,|\,\mathcal{D}\text{\,is a\,Stanley\,decomposition\,of}\,N\},$$ is called the Stanley depth of $N$. If $\mathfrak{m}:=(x_1,x_2,\dots,x_m)$ then the depth of $N$ is defined to be the common length of all maximal $N$-sequences in $\mathfrak{m}$. In \cite{ST} Stanley conjectured that $\sdepth (N)\geq \depth (N)$. This conjecture was later disproved by Duval et al. \cite{D} in $2016$. Stanley depth has been studied extensively in the last two decades see for example \cite{MI,IQ,pop,FA,4}.
Let $I$ be a monomial ideal of $S$. It is known in general that the depth of the powers of $I$, $\depth(S/I^t)$, stabilize
for large $t$. Indeed this follows from the general theorems that apply to any graded ideal
of $S$. In particular, by \cite{B} $\min\{\depth(S/I^t)\}\leq n-l(I)$, where $l(I)$ is the analytic
spread of $I$, and the minimum is taken over all powers $t$. In \cite{Br}, Brodmann showed
that for sufficiently large $t$, $\depth(S/I^t)$ is a constant, and this constant is bounded
above by $n-l(I)$. However, relatively little is known about $\depth(S/I^t)$ for specific values of $t$ other
than $t=1$. For some classes of powers of monomial ideals for which values or bounds are known we refer the readers to \cite{HH,LS,SM}.

Let $G$ be a finite, undirected and simple graph on $m$ vertices $v_1,v_2,\dots,v_m$. The \emph{edge ideal} $I(G)$ of the graph $G$ is the ideal of $S$ generated by all monomials of the form $x_ix_j$ such that $\{v_i,v_j\}$ is an edge of $G$. Let $m\geq 2$. A \emph{path} $P_m$ of length $m-1$ is a graph on $m$ vertices such that the vertex set of $P_m$ can be ordered in a way that whenever two vertices are consecutive in the list, there is an edge between them. A \emph{tree} is a graph in which any two vertices are connected by exactly one path. The \emph{diameter} of a connected graph $G$ is the maximum distance between any two vertices, where the distance between two vertices is given by the minimum length of a path connecting the vertices. For $t\geq 1$, Morey gave a lower bound for depth of $S/I^t(T)$ when $T$ is a tree in \cite{SM}, in terms of the diameter of $T$. Later on, in \cite{FA} Pournaki et. al. proved that this lower bound also serves as a lower bound for Stanley depth of $S/I^t(T)$. This lower bound being dependent on the diameter of a tree is weak in general.

The main focus of this paper is to give a better lower bound for some classes of trees. These bounds are independent of the diameters of the trees we considered and are much better than the bounds given in \cite{SM,FA}. Note that the lower bound for the depth of an edge ideal of a tree also provide a lower bound on the power for which the depth stabilizes.
Our work encompasses the computation of lower bounds for depth and Stanley depth of the powers of the edge ideals associated with some classes of caterpillar and lobster trees. The lower bound for the caterpillar trees depends on the power of the edge ideal, the number of leaves and the order of the path, see Theorem \ref{th2} and Corollary \ref{cor3}, while for the lobster trees it depends upon the power of the edge ideal and the number of near leaves, see Theorem \ref{thss} and Corollary \ref{cor4}. These parameters collectively make much sharper bounds than the bounds given in \cite{SM,FA}. We gratefully acknowledge the use of the computer algebra system CoCoA (\cite{COC}).
	\section{Definitions and Notations}
We start this section with a review of some notations and definitions, for more details, see \cite{JA,V}. Note that by abuse of notation, $x_i$ will at times be used to denote both a vertex of a graph $G$ and the corresponding variable of the polynomial ring $S$. Let $G$ be a graph with $V(G):=\{x_1,x_2,\dots,x_m\}$ and edge set $E(G)$. For a vertex $x_i$ of $G$ the set $N(x_i):=\{x_j \,|\, x_ix_j \in E(G) \}$ is called the \emph{neighborhood} of the vertex $x_i$. A vertex $x_i$ is called a \textit{leaf} (or \emph{pendant vertex}) if $N(x_i)$ has cardinality one and $x_i$ is called \textit{isolated} if $N(x_i)=\emptyset$. The \emph{parity} of an integer is its attribute of being even or odd. A graph with one vertex and no edges is called a \emph{trivial graph}. An \emph{internal vertex} is a vertex in a tree which is not a leaf. A graph with one internal vertex and $k$ leaves is called a $k$-\emph{star}, denoted by $\mathcal{S}_k$. Note that $\mathcal{S}_0$ is a trivial graph. A \emph{caterpillar tree} is a tree in which the removal of all pendant vertices results in a path. A \emph{lobster tree} is a tree with the property that the removal of pendant vertices leaves a caterpillar.
		
		\begin{Definition}\label{def1}
		{\em Let $n\geq 1$ and $k\geq 2$ be integers and $P_n$ be a path on $n$ vertices $\{u_1,u_2,\dots,u_n\}$ that is $E(P_n)=\{u_iu_{i+1}:1\leq i\leq n-1\}$ (for $n=1$, $E(P_n)=\emptyset$). We define a graph on $nk$ vertices by attaching $k-1$ pendant vertices at each $u_i$. We denote this graph by ${P}_{n,k}$.}
	\end{Definition}
	\begin{figure}[h!]
		\centering		\includegraphics[width=8cm]{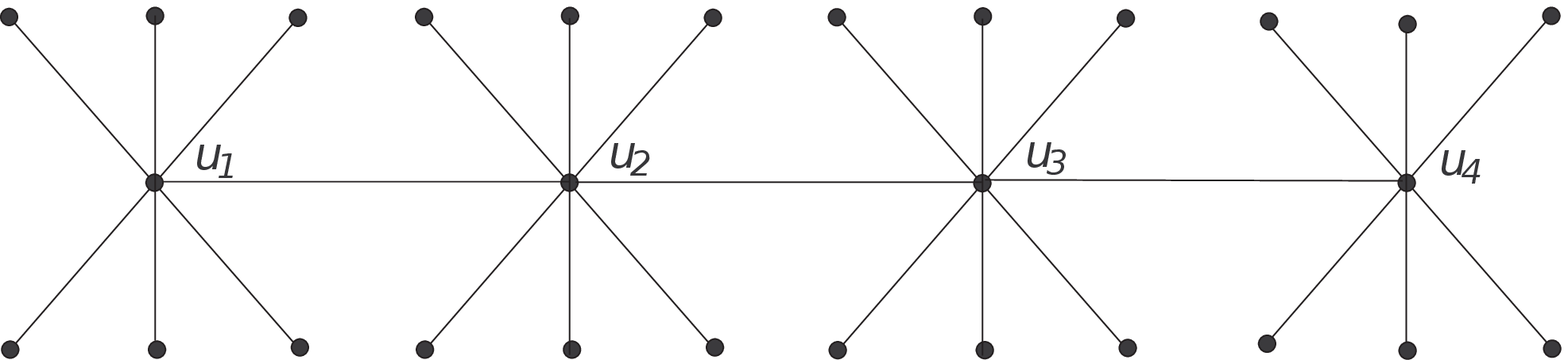}
		\caption{$P_{4,7}$}\label{fig1}
	\end{figure}
\noindent For example of $P_{n,k}$ see Fig. \ref{fig1}. \\

Let $n\geq 2$, $k\geq2$ and $l\geq 1$ be integers with $l\in [k]:=\{1,2,\dots,k\}$. Let $P_{n,k,l}$ be a graph which is obtained by removing $k-l$ pendant vertices attached to the vertex $u_n$ of the graph $P_{n,k}$. Note that $P_{n,k,k}=P_{n,k}$. For examples of $P_{n,k,l}$ see Fig. \ref{fig2}. It is easy to see that $P_{n,k,l}$ belongs to the family of caterpillar graphs.
\begin{Remark}
{\em In Definition \ref{def1}, $P_{1,k}$ represents a $(k-1)$-star.
}
\end{Remark}
\begin{figure}[h!]
	\centering
	\includegraphics[width=15cm]{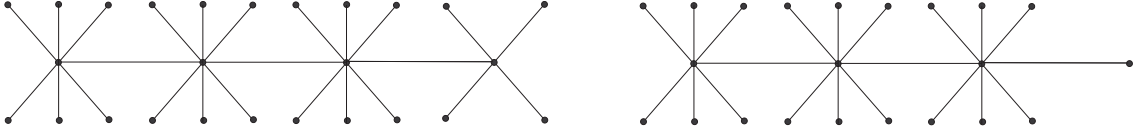}	
	\caption{From left to right, $P_{4,7,5}$ and $P_{4,7,1}$, respectively.}	\label{fig2}
\end{figure}
\begin{Definition}{\em
		Let $r\geq 2$ and $p\geq 1$ be integers. Let $\mathcal{S}_r$ be a star on $r+1$ vertices say $\{v_1,v_2,\dots, v_{r},v_{r+1}\}$ with $v_{r+1}$ as a central vertex. We define a graph by adding $p$ pendant vertices to each vertex $v_i$ with $1\leq i\leq r$. We denote this graph by $\mathcal{S}_{r,p}$.}
\end{Definition}
\begin{figure}[h!]
		\centering		\includegraphics[width=4cm]{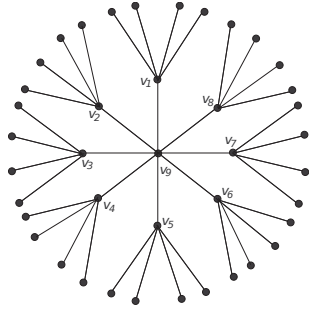}
		\caption{$\mathcal{S}_{8,4}$}\label{fig31}
	\end{figure}
\noindent For example of $\mathcal{S}_{r,p}$ see Fig. \ref{fig31}.\\

Let $q\geq 0$ with $q\leq p$ be an integer, then $\mathcal{S}_{r,p,q}$ is a graph which is obtained by removing $p-q$ leaves from exactly one $v_i$. Clearly $\mathcal{S}_{r,p,p}=\mathcal{S}_{r,p}$. For examples of $\mathcal{S}_{r,p,q}$ see Fig. \ref{fig4}. \\
\begin{figure}[h!]
		\centering		\includegraphics[width=8.5cm]{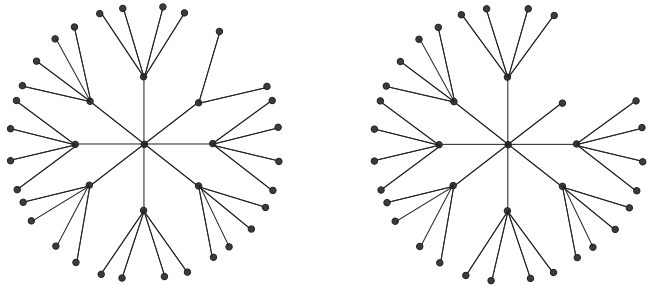}
		\caption{From left to right, $\mathcal{S}_{8,4,2}$ and $\mathcal{S}_{8,4,0}$, respectively}\label{fig4}
	\end{figure}

In order to make the paper self contained we recall some known results that we use in this paper.

\begin{Lemma}[{\cite[Proposition 1.2.9]{BH}}] (\text{Depth Lemma})
		If
		$0
		\longrightarrow\ A_1 \longrightarrow\ A_2
		\longrightarrow\ A_3 \longrightarrow\ 0,$
		is a short exact sequence of $\mathbb{Z}^m$-graded $S$-modules, then
		\begin{enumerate}
			\item $\depth A_2\geq \min\{\depth A_1,\depth A_3\},$
			\item $\depth A_1\geq \min\{\depth A_2,\depth A_3+1\},$
			\item $\depth A_3\geq \min\{\depth A_2,\depth A_1-1\}.$
		\end{enumerate}	
\end{Lemma}
\noindent A. Rauf proved the following lemma for Stanley depth.
\begin{Lemma}[{\cite[Lemma 2.2]{4}}] \label{le9}
		If
		$0
		\longrightarrow\ A_1 \longrightarrow\ A_2
		\longrightarrow\ A_3 \longrightarrow\ 0,$
		is a short exact sequence of $\mathbb{Z}^m$-graded $S$-modules, then
		$$\sdepth A_2\geq \min\{\sdepth A_1,\sdepth A_3\}.$$
	\end{Lemma}
\begin{Lemma}[{\cite[Lemma 3.6]{HV}}] \label{le15}
		Let $I$ be a monomial ideal of $S$. If $S'=S[y]$ is the polynomial ring over $S$ in the variable $y$, then $\depth (S'/IS')=\depth(S/I)+1$ and $\sdepth (S'/IS')=\sdepth(S/I)+1.$
	\end{Lemma}
\begin{Lemma} [{\cite[Lemma 1.1]{pop}}] \label{ddisjoint}
Let $I \subset K[x_1, \dots,x_r]=S_1$ and $ J \subset K[x_{r+1}, \dots , x_m] =S_2$ be monomial ideals, where $1 < r<m$. Then $$ \depth_S (S/ (IS+JS))= \depth_{S_1}(S_1/I) + \depth_{S_2}(S_2/J).$$
\end{Lemma}
\begin{Lemma} [{\cite[Theorem 3.1]{4}}] \label{sdisjoint}
	Let $I \subset K[x_1, \dots,x_r]=S_1$ and $ J \subset K[x_{r+1}, \dots , x_m] =S_2$ be monomial ideals, where $1 < r<m$. Then $$ \sdepth_S(S/ (IS+JS)) \geq \sdepth_{S_1}(S_1/I) + \sdepth_{S_2}(S_2/J).$$
\end{Lemma}
\noindent The following two lemmas play a key role in the proofs of our main theorems.	
\begin{Lemma}[{\cite[Lemma 2.10]{SM}}] \label{le2}
		Let $G$ be a graph and $I=I(G)$. Let $x_i$ be a leaf of $G$ and $x_j$ be the unique neighbor of $x_i$. Then $(I^t:x_ix_j)=I^{t-1}$ for any $t\geq2.$
	\end{Lemma}
\begin{Lemma}[{\cite[Lemma 2.5]{SM}}]\label{morey}
Let $I$ be a square-free monomial ideal in a polynomial ring $S$ and let $M$ be a monomial in $S$. If $y$ is a variable such that $y$ does not divide $M$ and $J$ is the extension in $R$ of the minor of $I$ formed by setting $y=0$, then $((I^t:M),y)=((J^t:M),y)$ for any $t\geq 1$.
 \end{Lemma}

\begin{Proposition} [{\cite[Theorem 2.6 and 2.9]{AA}}] \label{2}
		If $I=I(\mathcal{S}_{m-1})$, which is a square-free monomials ideal of $S$, then
		$\depth(S/I) = \sdepth(S/I) = 1$ and 	$\depth(S/I^t) , \sdepth(S/I^t) \geq 1.$
\end{Proposition}
\begin{Lemma} [{\cite[Lemma 2.6]{SM}}]  \label{le7}
		Let $G$ be a bipartite graph and $I=I(G)$. Then for all $t \geq 1$,
		$$\depth(S/I^t) \geq 1.$$
\end{Lemma}
\begin{Theorem}[{\cite[Theorem 1.4]{FF}}]\label{le12}
For a finitely generated $\mathbb{Z}^m$-graded $S$-module $N$, if $\sdepth(N)=0$ then $\depth(N)=0$.
\end{Theorem}
A forest is a graph with each connected component a tree. The following theorems give lower bounds for depth and Stanley depth of powers of an edge ideal corresponding to a forest.
\begin{Theorem} [{\cite[Theorem 3.4]{SM}}]  \label{rth1}
	Let $G$ be a forest having $s$ number of connected components $G_1$, $G_2$ , \dots , $G_s$. Let $I=I(G)$ and $d_j$ be the diameter of $G_j$ and suppose $d=\max\limits_j \{d_j\} $. Then for $t \geq 1$ $$\depth(S/I^t) \geq \max \big \{\big \lceil \frac{d-t+2}{3} \big \rceil +s-1 , s \big \}.$$
\end{Theorem}	
\begin{Theorem} [{\cite[Theorem 2.7]{FA}}]  \label{rth2}
	Let $G$ be a forest having $s$ number of connected components $G_1$, $G_2$ , \dots , $G_s$. Let $I=I(G)$ and $d_j$ be the diameter of $G_j$ and suppose $d=\max\limits_j \{d_j\} $. Then for $t \geq 1$ $$\sdepth(S/I^t) \geq \max \big \{\big \lceil \frac{d-t+2}{3} \big \rceil +s-1 , s \big \}.$$
\end{Theorem}
Let $v$ be a vertex of $G$, $v$ is called a \emph{near leaf} of $G$ if $v$ is not a leaf and $N(v)$ contains at most one vertex that is not a leaf. Let $a$ denote the number of near leaves of $G$. The bounds for depth and Stanley depth are strengthened by the following results in the same papers.
\begin{Corollary} [{\cite[Corollary 3.7]{SM}}]  \label{rth111}
	Let $G$ be a forest having $s$ number of connected components $G_1$, $G_2$ , \dots , $G_s$. Let $I=I(G)$ and $d_j$ be the diameter of $G_j$ and suppose $d=\max\limits_j \{d_j\} $, and let $a$ be the number of near leaves of a component of diameter $d$. Then for $t \geq 1$ $$\depth(S/I^t) \geq \max \big \{\big \lceil \frac{d-t+a}{3} \big \rceil +s-1 , s \big \}.$$
\end{Corollary}
\begin{Corollary} [{\cite[Corollary 3.2]{FA}}]  \label{rth222}
	Let $G$ be a forest having $s$ number of connected components $G_1$, $G_2$ , \dots , $G_s$. Let $I=I(G)$ and $d_j$ be the diameter of $G_j$ and suppose $d=\max\limits_j \{d_j\} $, and let $a$ be the number of near leaves of a component of diameter $d$. Then for $t \geq 1$ $$\sdepth(S/I^t) \geq \max \big \{\big \lceil \frac{d-t+a}{3} \big \rceil +s-1 , s \big \}.$$
\end{Corollary}
If $T$ is a tree, then the following corollary is an immediate consequence of the Corollary \ref{rth111} and \ref{rth222}.
\begin{Corollary}\label{dsc}
Let $T$ be a tree and $d$ be the diameter of $T$ and let $a$ be the number of near leaves of $T$. If $I=I(T)$, then for $t \geq 1$ $$\depth(S/I^t),\sdepth(S/I^t) \geq \max\big\{\lceil \frac{d-t+a}{3} \big \rceil,1\big\} .$$
\end{Corollary}
The bound in Corollary \ref{dsc} depends on the diameter of  $T$ and the number of near leaves in $T$.
If $I$ is the edge ideal of $P_{n,k}$ or $\mathcal{S}_{r,p}$, we give lower bounds for depth and Stanley depth of $S/I^t$ as Corollary \ref{cor3} and Corollary \ref{cor4}. We observe that our bounds are much sharper than the bounds given in Corollary \ref{dsc}.

	\section{Powers of Edge Ideal of a subclass of Caterpillar tree}
\noindent Let $n,k\geq 2$ and $l\in [k]$. We define $A_i:=\{y_{1i},y_{2i},\dots,y_{(k-1)i}\}$, for $1\leq i\leq n-1$, and $A_n:=\{y_{1n},y_{2n},\dots,y_{(l-1)n}\}$, where $A_n=\emptyset$ if $l=1$. Let $\bar{A_i}:=\{u_i\}\cup A_i$, $\bar{A}_n:=\{u_n\}\cup A_n$ and $A:=\bar{A}_1\cup \bar{A}_2\cup\dots\cup \bar{A}_n$. Let $S$ be the polynomial ring over a field $K$ in variables of set $A$ that is  $S:=K[A]$. Let $I=I(P_{n,k})$, in this section we give lower bounds for depth and Stanley depth of $S/I^t$ for $t\geq 1$. We denote by $G(I)$, the minimal set of monomial generators of the monomial ideal $I$. If $l\geq 2$, then
\begin{figure}[h!]
		\centering		\includegraphics[width=10cm]{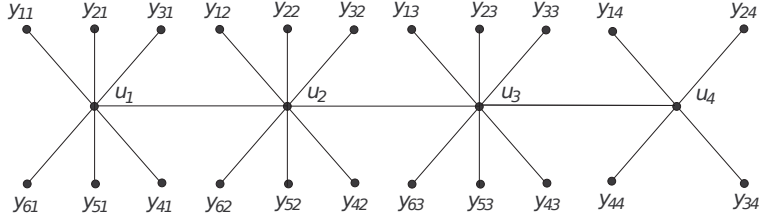}
		\caption{Graph $P_{4,7,5}$ with labelled vertices.}
	\end{figure}

$$G(I(P_{n,k,l}))= \bigcup_{i=1}^{n-1}\big\{u_iu_{i+1},u_i y_{1i},u_i y_{2i},\dots,u_iy_{(k-1)i}\big\}\cup \big\{u_n y_{1n},u_n y_{2n},\dots,u_ny_{(l-1)n}\big\}.$$ 
 If $l=1$, then
$$G(I(P_{n,k,1}))= \bigcup_{i=1}^{n-1}\big\{u_iu_{i+1},u_i y_{1i},u_i y_{2i},\dots,u_iy_{(k-1)i}\big\}.$$ Note that $P_{n,k,k}=P_{n,k}$. Also for $1\leq j\leq n-1$, we have $$I(P_{j,k}):=I(P_{n,k,l})\cap K[\bar{A}_1\cup  \dots\cup\bar{A}_{j-1} \cup \bar{A}_{j}].$$

\begin{Lemma}\label{le6}
If $I=I(P_{n,k,l})$ then for $t\geq 1,$ $ (I^t , u_n)= (I^t(P_{n-1,k}),u_n).$
\end{Lemma}	
\begin{proof}
The inclusion $(I^t(P_{n-1,k}),u_n)\subseteq (I^t,u_n)$ is clear. Conversely, if $u\in I^t$ is a monomial which
is not divisible by $u_n$, then, by the definition of $G(I)$, it follows that $u\in I^t(P_{n-1,k})$.

\end{proof}
\begin{Remark}
{\em Let $t\geq 1$. From Proposition \ref{2} it follows that $\depth(S/I(P_{1,k}))=\sdepth(S/I(P_{1,k}))=1$ and $\depth(S/I^t(P_{1,k})),\sdepth(S/I^t(P_{1,k}))\geq 1$.}
\end{Remark}
\begin{Lemma}\label{le8}
Let $k\geq 2$, $l\in [k]$ and $I=I(P_{2,k,l})$. We have that $$\depth(S/I) \, , \,\sdepth(S/I)\geq l.$$
\end{Lemma}
\begin{proof}
	Clearly $u_1$ and $u_2$ have $k-1$ and $l-1$ pendant vertices respectively. Consider the short exact sequence:
	$$0\longrightarrow\ S/(I:u_2) \xrightarrow{\,\cdot u_2\,}\ S/I
	\longrightarrow\ S/(I,u_2) \longrightarrow\ 0.$$
	Now $(I:u_2) = (x:x\in N(u_2))$ and $S/(I:u_2)\cong K[A_1\cup \{u_2\}]$, thus $\depth(S/(I:u_2))=k$.
	$S/(I,u_2)\cong (K[\bar{A}_1]/I(P_{1,k}))[A_2]$. Thus by Lemma \ref{le15} and Proposition \ref{2} $\depth(S/(I,u_2))=1+l-1=l$.
	By Depth Lemma,
	$\depth(S/I)\geq \min\{\depth(S/(I:u_2)),
	\depth(S/(I,u_2))\} \geq l.$
	Proof for Stanley depth is similar using Lemma \ref{le9}.
\end{proof}
\begin{Proposition} \label{th13}
	Let $n,k\geq 2$, $l\in [k]$ and $I=I(P_{n,k,l})$. We have that
	$$\depth(S/I),\sdepth(S/I) \geq \left\{\begin{array}{ll}	\Big(\frac{n-2}{2}\Big)k+l\,, & \text{if\,\, $n$ is even;} \\
	\Big(\frac{n-1}{2}\Big)k+ 1\,, & \text{if\,\, $l\geq2$ and $n$ is odd;}\\
	\Big(\frac{n-1}{2}\Big)k\,, & \text{if\,\, $l=1$ and $n$ is odd.}
	\end{array}\right.$$
\end{Proposition}
\begin{proof}
For $n=2$, the conclusion follows from Lemma \ref{le8}. For $n=3$, we consider the following short exact sequence
$$0\longrightarrow\ S/(I:u_3) \xrightarrow{\,\cdot u_3\,}\ S/I
	\longrightarrow\ S/(I,u_3) \longrightarrow\ 0.$$
Now $$S/(I:u_3) \cong S/ ((y:y \in N(u_3))+I(P_{1,k}))\cong \big(K[\bar{A}_1]/I(P_{1,k})\big)[A_2\cup \{u_3\}].$$ From Lemma \ref{le15} and Proposition \ref{2} it follows that $$\depth S/(I:u_3)= k+\depth (K[\bar{A}_1]/I(P_{1,k}))=k+1.$$
Now, since $S/(I,u_3) \cong S/(I(P_{2,k}),u_3)\cong \big(K[\bar{A}_1\cup \bar{A}_2]/(I(P_{2,k})\big)[A_3]$, by Lemma \ref{le15} and Lemma \ref{le8} it follows that $$\depth (S/(I,u_3)) = l-1+\depth (K[\bar{A}_1\cup \bar{A}_2]/(I(P_{2,k})) \geq l-1+k=k+l-1.$$
	So by Depth Lemma
	$$\depth(S/I) \geq \left\{\begin{array}{ll}	k+1\,, & \text{if\,\, $l\geq 2$;} \\
	k\,, & \text{if\,\, $l=1$.}
	\end{array}\right.$$
Similarly, from Lemma \ref{le9} it follows that
$$\sdepth(S/I) \geq \left\{\begin{array}{ll}	k+1\,, & \text{if\,\, $l\geq 2$;} \\
	k\,, & \text{if\,\, $l=1$.}
	\end{array}\right.$$
For $n \geq 4$, we consider the following short exact sequence
$$0\longrightarrow\ S/(I:u_n) \xrightarrow{\,\cdot u_n\,}\ S/I
	\longrightarrow\ S/(I,u_n) \longrightarrow\ 0.$$
Notice that $$S/(I:u_n) = S/((y:y \in N(u_n))+I(P_{n-2,k}))\cong \big(K[A\backslash(\bar{A}_{n-1}\cup \bar{A_n})]/I(P_{n-2,k})\big)[A_{n-1}\cup u_n].$$
and $$S/(I,u_n)\cong \big(K[A\backslash\bar{A}_n]/I(P_{n-1,k})\big)[A_n].$$	
\textbf{Case 1:} $n$ is even. \\
	 For $k\geq 2$, since $P_{n-2,k}=P_{n-2,k,k}$ and $P_{n-1,k}=P_{n-1,k,k}$,  using induction on $n$ and Lemma \ref{le15}, it follows that:	
	\begin{multline*}
		\depth(S/(I:u_n)) =\depth\big(K[A\backslash(\bar{A}_{n-1}\cup \bar{A_n})]/I(P_{n-2,k})\big)+(k-1)+1\geq \\ \big( \big(\frac{n-2-2}{2} \big )k+k\big) + k= \big (\frac{n}{2} \big )k,
	\end{multline*}
and
\begin{multline*}	$$
		\depth(S/(I,u_n))=\depth \big(K[A\backslash\bar{A}_n]/I(P_{n-1,k})\big)+(l-1) \geq\\ \big(\big (\frac{n-2}{2} \big )k +1\big) + (l-1) = \big (\frac{n-2}{2} \big )k +l.
$$
\end{multline*}
	Thus by Depth Lemma,
	$$\depth(S/I) \geq \big (\frac{n-2}{2} \big )k +l.$$
	\textbf{Case 2:} When $n$ is odd.\\ Again by induction on $n$ and Lemma \ref{le15},
\begin{multline*}
		\depth(S/(I:u_n)) =\depth\big(K[A\backslash(\bar{A}_{n-1}\cup \bar{A_n})]/I(P_{n-2,k})\big)+(k-1)+1\geq \\ \big( \big(\frac{n-3}{2} \big )k+1\big) + k= \big (\frac{n-1}{2} \big )k+1,
	\end{multline*}	
and
\begin{multline*}	
$$\depth(S/(I,u_n))=\depth \big( K[A\backslash\bar{A}_n]/I(P_{n-1,k})\big)+(l-1) \geq\\ \big(\big( \frac{n-1-2}{2} \big )k +k\big)+ (l-1)= \big (\frac{n-1}{2} \big )k + (l-1).$$
\end{multline*}
Thus by Depth Lemma,
		$$\depth(S/I) \geq \left\{\begin{array}{ll}
	\big(\frac{n-1}{2}\big)k+ 1\,, & \text{if\,\, $l\geq2$;}\\
	\big(\frac{n-1}{2}\big)k\,, & \text{if\,\, $l=1$.}
	\end{array}\right.$$
	Proof for the Stanley depth is similar by using Lemma \ref{le9} instead of Depth Lemma.
\end{proof}
\begin{Corollary}\label{sq}
	If $n\geq2$, $k\geq 2$ and $I=I(P_{n,k})$ then
	$$\depth(S/I),\sdepth(S/I) \geq \left\{\begin{array}{ll}	\big(\frac{n}{2}\big)k\,, & \text{if\,\, $n$ is even;} \\
	\big(\frac{n-1}{2}\big)k+ 1\,, & \text{if\,\, $n$ is odd.}
	\end{array}\right.$$
\end{Corollary}
\begin{Example}{\em
By using CoCoA (for sdepth we use SdepthLib.coc \cite{GR}) it has been noticed that the equality may hold in some cases. For instance,  $\depth(S/I(P_{4,4}))=\sdepth(S/I(P_{4,4}))=8=\big(\frac{4}{2}\big)4$, and $\depth(S/I(P_{5,3}))=\sdepth(S/I(P_{5,3}))=7=\big(\frac{5-1}{2}\big)3+1$.
}
\end{Example}
For convenience we label the vertices of $A_n=\{y_{1n},y_{2n},\dots,y_{(l-1)n}\}$ by $s_1,s_2,\dots,s_{l-1}$. Set $S_i:=K[A]/(s_1,s_2,\dots, s_i)$ and $I_i:=I\cap S_{i}.$
\begin{Theorem} \label{th2}
	Let $n,k\geq 2$, $l\in [k]$, $t\geq1,$ and $I=I(P_{n,k,l})$. We have that
	$$\depth(S/I^t),\sdepth(S/I^t) \geq \left\{\begin{array}{ll}\max\Big\{1,\Big(\frac{n-t-1}{2}\Big)k+l-1\Big\} \,, & \text{if\,\, $n$ and $t$ have opposite parity;} \\
	
	\max\Big\{1,\Big(\frac{n-t}{2}\Big)k\Big\} \,, & \text{if\,\, $n$ and $t$ have the same parity} \\
	                              &\text{and $2\leq l\leq k$};\\
	\max\Big\{1,\Big(\frac{n-t}{2}\Big)k-1\Big\} \,, & \text {if\,\, $n$ and $t$ have the same parity}\\
	                              &\text{and $l=1$.}
	\end{array}\right.$$
	
	\begin{proof}
		Since $P_{n,k,l}$ is a bipartite graph, from Lemma \ref{le7} it follows that $\depth(S/I^t)\geq 1$ for all $t\geq 1$. We use induction on $n$ and $t.$ For $n\geq 2$ and $t=1$, the result follows from Proposition \ref{th13}. For $n=2$ and $t\geq 1$, the result follows from Lemma \ref{le7}. Let $n=3$. For $t\geq 3$, the result again follows from Lemma \ref{le7}. If $t=2$ then we need to prove the desired inequality. Let $I=I(P_{3,k,l})$. We will prove that
$$\depth(S/I^2)  \geq \max\Big\{1,\Big(\frac{3-2-1}{2}\Big)k+l-1\Big\}=\max\{1,l-1\}.$$
If $l=1$, then $\max\{1,l-1\}=1$ and from Lemma \ref{le7} we have that $\depth(S/I^2) \geq 1.$ Assume that $l\geq 2$ and consider the following short exact sequence
$$0\longrightarrow\ S/(I^2:u_3) \xrightarrow{\,\cdot u_3\,}\ S/I^2
\longrightarrow\ S/(I^2,u_3) \longrightarrow\ 0.$$
By Lemma \ref{le6},
$S/(I^2,u_3) \cong S/(I^2(P_{2,k}), u_3)\cong (K[A \backslash (\bar{A}_3)] /I^2(P_{2,k}) )[A_3] .$	Therefore,
$$\depth(S/(I^2,u_3)) \geq 1 +(l-1) =l.$$
We consider the following family of short exact sequences:
	$$0\longrightarrow\
	S_0/(I_0^2:u_3s_1) \xrightarrow{\,\cdot s_1\,} S_0/(I_0^2:u_3) \longrightarrow\
	S_0/((I_0^2:u_3),s_1) \longrightarrow\ 0,$$
	$$0\longrightarrow\ S_1/(I_1^2:u_3s_2) \xrightarrow{\,\cdot s_2\,}\ S_1/(I_1^2:u_3) \longrightarrow\
	S_1/((I_1^2:u_3),s_2) \longrightarrow\ 0,$$
	$$\vdots$$
	$$0\longrightarrow\
	S_{l-2}/(I_{l-2}^2:u_3s_{l-1}) \xrightarrow{\,\cdot s_{l-1}\,}\ S_{l-2}/(I_{l-2}^2:u_3)
	\longrightarrow\ S_{l-2}/((I_{l-2}^2:u_3),s_{l-1}) \longrightarrow\ 0.$$
	By Lemma \ref{le2}, $\depth(S_i/(I_{i}^2:u_3s_{i+1}))=\depth(S_i/I_{i})$ and by Proposition \ref{th13}
	$$\depth(S_i/(I_{i}^2:u_3s_{i+1})) \geq k+1.$$
		Since $S_{l-2}/((I_{l-2}^2:u_3),s_{l-1})\cong S_{l-1}/(I_{l-1}^2:u_3)$, consider the following short exact sequence		
	$$0\longrightarrow\
	S_{l-1}/(I_{l-1}^2:u_3u_2) \xrightarrow{\,\cdot u_2\,}\ S_{l-1}/(I_{l-1}^2:u_3)
	\longrightarrow\ S_{l-1}/((I_{l-1}^2:u_3),u_2) \longrightarrow\ 0.$$
	By Lemma \ref{le2}, $\depth(S_{l-1}/(I_{l-1}^2:u_3u_2))=	\depth(S_{l-1}/I_{l-1})$, here $l=1$ and by Proposition \ref{th13}
	$$\depth(S_{l-1}/(I_{l-1}^2:u_3u_2))\geq
k.$$
	Clearly $S_{l-1}/((I_{l-1}^2:u_3),u_2)\cong \Big(K[A\backslash (\bar{A}_{2}\cup \bar{A_3})]/I^2 (P_{1,k}))\Big)[A_{2}\cup\{u_3\}]$,  therefore by Lemma \ref{le15} and Proposition \ref{2} we have
	$$\depth(S_{l-1}/((I_{l-1}^2:u_3),u_2))= \depth\Big(K[A\backslash (\bar{A}_{2}\cup \bar{A_3})]/I^2 (P_{1,k})\Big)+(k-1)+1 \geq k+1.$$  Depth Lemma implies,
	$$\depth(S/I^2) \geq  l.$$
Now let $n\geq 4$, $t\geq 2$ and $I=I(P_{n,k,l})$. We consider two cases:\\
		\textbf{Case 1:} When $n$ and $t$ have the same parity.\\
		\textbf{(a).} Let $l=1$. Consider the following short exact sequence
		$$0\longrightarrow\ S/(I^t:u_n) \xrightarrow{\,\cdot u_n\,}\ S/I^t
		\longrightarrow\ S/(I^t,u_n) \longrightarrow\ 0.$$
		By Lemma \ref{le6},
		$S/(I^t,u_n)=S/(u_n,I^t(P_{n-1,k})).$
		For $k\geq 2$, since $P_{n-1,k}=P_{n-1,k,k}$, $n-1$ and $t$ have the opposite parity, using induction on $n$, it follows that:
		$$\depth(S/(I^t,u_n)) \geq  \Big(\frac{n-1-t-1}{2}\Big)k+(k-1) =  \Big(\frac{n-t}{2}\Big)k-1.$$
		We consider another short exact sequence as follows
		$$0\longrightarrow\ S/(I^t:u_nu_{n-1}) \xrightarrow{\,\cdot u_{n-1}\,}\ S/(I^t:u_n)
		\longrightarrow\ S/((I^t:u_n),u_{n-1}) \longrightarrow\ 0.$$
		Since $u_{n-1}$ is the unique neighbor of $u_n$, from Lemma \ref{le2} it follows that $(I^t:u_nu_{n-1})=I^{t-1}$. Now $n$ and $t-1$ have the opposite parity thus by induction on $t$
		$$\depth(S/(I^t:u_nu_{n-1})) = \depth(S/I^{t-1})\geq \Big(\frac{n-(t-1)-1}{2}\Big)k+1-1=\Big(\frac{n-t}{2}\Big)k, $$
		and by Lemma \ref{morey} we have $$S/((I^t:u_n),u_{n-1})\cong \big(K[A\backslash (\bar{A}_{n-1}\cup \bar{A_n})]/I^t(P_{n-2,k})\big)[A_{n-1}\cup \{u_n\}].$$ By induction on $n$ and Lemma \ref{le15}
		$$\depth(S/((I^t:u_n),u_{n-1})\geq \Big(\big(\frac{n-t-2}{2}\big)k\Big)+k \geq \Big(\frac{n-t}{2}\Big)k.$$
		Thus by Depth Lemma we have,
		$$\depth (S/I^t) \geq  \Big(\frac{n-t}{2}\Big)k -1.$$
		\textbf{(b).} Let $l\geq 2$. Consider the following short exact sequence
		$$0\longrightarrow\ S/(I^t:u_n) \xrightarrow{\,\cdot u_n\,}\ S/I^t
		\longrightarrow\ S/(I^t,u_n) \longrightarrow\ 0.$$
		By Lemma \ref{le6},
		$S/(I^t,u_n)\cong\big(K[A\backslash \bar{A_n}]/(I^t(P_{n-1,k}))\big)[A_n].$ Therefore by Lemma \ref{le15} we have
		$$\depth(S/(I^t,u_n)) = \depth (K[A\backslash \bar{A_n}]/(I^t(P_{n-1,k}))+ |A_n|$$
		Here $n-1$ and $t$ have the opposite parity, so by induction on $n$,
	     $$\depth(S/(I^t,u_n)) \geq \Big(\big(\frac{n-t-2}{2}\big)k+k-1\Big)  +(l-1)
			\geq \Big(\frac{n-t}{2}\Big)k+l-2 .$$
Now we find lower bound for depth of module $S/(I^t:u_n)$. Let $0\leq i\leq l-2$. By Lemma \ref{morey}, $S_{i}/((I_{i}^t:u_n),s_{i+1})\cong S_{i+1}/(I_{i+1}^t:u_n)$ where $S_0=S$ and $I_0=I$.
  We consider the following family of short exact sequences:
		$$0\longrightarrow\
		S_0/(I_0^t:u_ns_1) \xrightarrow{\,\cdot s_1\,}\ S_0/(I_0^t:u_n) \longrightarrow\
		S_0/((I_0^t:u_n),s_1) \longrightarrow\ 0,$$
		$$0\longrightarrow\ S_1/(I_1^t:u_ns_2) \xrightarrow{\,\cdot s_2\,}\ S_1/(I_1^t:u_n) \longrightarrow\
		S_1/((I_1^t:u_n),s_2) \longrightarrow\ 0,$$
		$$\vdots$$
			$$0\longrightarrow\
		S_{l-2}/(I_{l-2}^t:u_ns_{l-1}) \xrightarrow{\,\cdot s_{l-1}\,}\ S_{l-2}/(I_{l-2}^t:u_n)
		\longrightarrow\ S_{l-2}/((I_{l-2}^t:u_n),s_{l-1}) \longrightarrow\ 0.$$
	    By Lemma \ref{le2}, $\depth(S_i/(I_{i}^t:u_ns_{i+1}))=\depth(S_i/I_{i}^{t-1}).$
Here $n$ and $t-1$ have the opposite parity so, by using induction on $t$,
		$$\depth(S_i/(I_{i}^t:u_ns_{i+1}))\geq
			\Big(\frac{n-(t-1)-1}{2}\Big)k+(l-1-i)
			\geq \Big(\frac{n-t}{2}\Big)k.$$
Again by Lemma \ref{morey}$, S_{l-2}/((I_{l-2}^t:u_n),s_{l-1})\cong S_{l-1}/(I_{l-1}^t:u_n)$. Now consider the following short exact sequence		
$$0\longrightarrow\
		S_{l-1}/(I_{l-1}^t:u_nu_{n-1}) \xrightarrow{\,\cdot u_{n-1}\,}\ S_{l-1}/(I_{l-1}^t:u_n)
		\longrightarrow\ S_{l-1}/((I_{l-1}^t:u_n),u_{n-1}) \longrightarrow\ 0,$$
by Lemma \ref{le2}, $\depth(S_{l-1}/(I_{l-1}^t:u_nu_{n-1}))=	\depth(S_{l-1}/I_{l-1}^{t-1})$. By using induction on $t$,
			$$\depth(S_{l-1}/(I_{l-1}^t:u_nu_{n-1}))\geq
			\Big(\frac{n-(t-1)-1}{2}\Big)k+\big(l-(l-1)\big)-1 = \Big(\frac{n-t}{2}\Big)k.$$
		Clearly $S_{l-1}/((I_{l-1}^t:u_n),u_{n-1})\cong S'/I_{l-1}^tS'$, where $S'=K[A\backslash(A_n\cup\{u_{n-1}\})]$. Thus $u_n$ and all variables in $A_{n-1}$ are regular on $S'/I_{l-1}^tS'$. Since $$S'/I_{l-1}^tS'\cong \Big(K[A\backslash (\bar{A}_{n-1}\cup \bar{A_n})]/I^t (P_{n-2,k}))\Big)[A_{n-1}\cup\{u_n\}],$$ therefore by Lemma \ref{le15}, we get
		$\depth(S'/I_{l-1}^tS')= \depth\Big(K[A\backslash (\bar{A}_{n-1}\cup \bar{A_n})]/I^t (P_{n-2,k})\Big)+(k-1)+1.$ Here $n-2$ and $t$ have the same parity, so by induction on $n$
	$$\depth(S'/I_{l-1}^tS') \geq \Big(\frac{n-2-t}{2}\Big)k+( k-1)+1= \Big(\frac{n-t}{2}\Big)k.$$
		Depth Lemma implies,
		$$\depth(S/I^t) \geq  \Big(\frac{n-t}{2}\Big)k.$$
\textbf{Case 2:} When $n$ and $t$ have the opposite parity.\\
		\textbf{(a).} Let $l=1$. For this consider the following short exact sequence:
		$$0\longrightarrow\ S/(I^t:u_n) \xrightarrow{\,\cdot u_n\,}\ S/I^t
		\longrightarrow\ S/(I^t,u_n) \longrightarrow\ 0.$$
		By Lemma \ref{le6},
		$S/(I^t,u_n)=S/(u_n,I^t(P_{n-1,k})).$
		Here $n-1$ and $t$ have the same parity, so by induction on $n$,
		$$\depth(S/(I^t,u_n)) \geq  \Big(\frac{n-t-1}{2}\Big)k.$$
		For the depth of module $S/(I^t:u_n)$, we consider another short exact sequence as follows:
		$$0\longrightarrow\ S/(I^t:u_nu_{n-1}) \xrightarrow{\,\cdot u_{n-1}\,}\ S/(I^t:u_n)
		\longrightarrow\ S/((I^t:u_n),u_{n-1}) \longrightarrow\ 0,$$
		since $u_{n-1}$ is the unique neighbor of $u_n$ thus by Lemma \ref{le2} we have $(I^t:u_nu_{n-1})=I^{t-1}$. Now $n$ and $t-1$ have the same parity thus by induction on $t$
\begin{multline*}		
$$\depth(S/(I^t:u_nu_{n-1})) = \depth(S/I^{t-1})\geq \Big(\frac{n-(t-1)}{2}\Big)k-1=\\\Big(\frac{n-t-1}{2}\Big)k+k-1>\Big(\frac{n-t-1}{2}\Big)k,$$
\end{multline*}	
and by Lemma \ref{morey} we have $$S/((I^t:u_n),u_{n-1})\cong \big(K[A\backslash (\bar{A}_{n-1}\cup \bar{A_n})]/I^t(P_{n-2,k})\big)[A_{n-1}\cup u_n],$$ by induction on $n$ and Lemma \ref{le15}
\begin{multline*}
\depth(S/((I^t:u_n),u_{n-1}))\geq \Big(\big(\frac{n-t-3}{2}\big)k+k-1\Big)+k=\\\big(\frac{n-t-1}{2}\big)k-1+k > \Big(\frac{n-t-1}{2}\Big)k.
\end{multline*}	
Thus by Depth Lemma we have,
		$$\depth S/I^t \geq  \Big(\frac{n-t-1}{2}\Big)k.$$
		\textbf{(b).} Let $l\geq 2$. Consider the short exact sequence
		$$0\longrightarrow\ S/(I^t:u_n) \xrightarrow{\,\cdot u_n\,}\ S/I^t
		\longrightarrow\ S/(I^t,u_n) \longrightarrow\ 0.$$
		By Lemma \ref{le6},
		$S/(I^t,u_n)\cong\big(K[A\backslash \bar{A_n}]/(I^t(P_{n-1,k}))\big)[A_n].$
		$$\depth(S/(I^t,u_n) = \depth \Big(K[A\backslash \bar{A_n}]/(I^t(P_{n-1,k}))\Big)+ |A_n|.$$
		Here $n-1$ and $t$ have the same parity, so by induction on $n$,
	     $$\depth(S/(I^t,u_n)) \geq \big(\frac{n-t-1}{2}\big)k+l-1.$$
Consider again the following family of short exact sequences:
		$$0\longrightarrow\
		S_0/(I_0^t:u_ns_1) \xrightarrow{\,\cdot s_1\,}\ S_0/(I_0^t:u_n) \longrightarrow\
		S_0/((I_0^t:u_n),s_1) \longrightarrow\ 0,$$
		$$0\longrightarrow\ S_1/(I_1^t:u_ns_2) \xrightarrow{\,\cdot s_2\,}\ S_1/(I_1^t:u_n) \longrightarrow\
		S_1/((I_1^t:u_n),s_2) \longrightarrow\ 0,$$
		$$\vdots$$
			$$0\longrightarrow\
		S_{l-2}/(I_{l-2}^t:u_ns_{l-1}) \xrightarrow{\,\cdot s_{l-1}\,}\ S_{l-2}/(I_{l-2}^t:u_n)
		\longrightarrow\ S_{l-2}/((I_{l-2}^t:u_n),s_{l-1}) \longrightarrow\ 0.$$
	    By Lemma \ref{le2},  $\depth(S_i/(I_{i}^t:u_ns_{i+1}))=\depth(S_i/I_{i}^{t-1}).$
Here $n$ and $t-1$ have the same parity so, by using induction on $t$,
		$$\depth(S_i/(I_{i}^t:u_ns_{i+1}))\geq
			\Big(\frac{n-(t-1)}{2}\Big)k =\Big(\frac{n-t+1}{2}\Big)k>\Big(\frac{n-t-1}{2}\Big)k+l-1.$$
Since $S_{l-2}/((I_{l-2}^t:u_n),s_{l-1})\cong S_{l-1}/(I_{l-1}^t:u_n)$, consider the following short exact sequence		
$$0\longrightarrow\
		S_{l-1}/(I_{l-1}^t:u_nu_{n-1}) \xrightarrow{\,\cdot u_{n-1}\,}\ S_{l-1}/(I_{l-1}^t:u_n)
		\longrightarrow\ S_{l-1}/((I_{l-1}^t:u_n),u_{n-1}) \longrightarrow\ 0.$$
By Lemma \ref{le2}, $\depth(S_{l-1}/(I_{l-1}^t:u_nu_{n-1}))=	\depth(S_{l-1}/I_{l-1}^{t-1})$, here $l=1$ and $n$ and $t-1$ have the same parity,  thus by induction on $t$,
			$$\depth(S_{l-1}/(I_{l-1}^t:u_nu_{n-1}))\geq
			\Big(\frac{n-(t-1)}{2}\Big)k-1\geq \Big(\frac{n-t-1}{2}\Big)k+l-1.$$
		Clearly $S_{l-1}/((I_{l-1}^t:u_n),u_{n-1})\cong S'/I_{l-1}^tS'$, where $S'=K[A\backslash(A_n\cup \{u_{n-1}\})]$. Thus $u_n$ and all variables in $A_{n-1}$ are regular on $S'/I_{l-1}^tS'$. Since $$S'/I_{l-1}^tS'\cong \Big(K[A\backslash (\bar{A}_{n-1}\cup \bar{A_n})]/I^t (P_{n-2,k})\Big)[A_{n-1}\cup\{u_n\}],$$ therefore
		$\depth(S'/I_{l-1}^tS')= \depth\Big(K[A\backslash (\bar{A}_{n-1}\cup \bar{A_n})]/I^t (P_{n-2,k})\Big)+(k-1)+1.$ Since $n-2$ and $t$ have the opposite parity, so by induction on $n$ and Lemma \ref{le15}, we get
	$$\depth(S'/I_{l-1}^tS') \geq \Big(\frac{n-2-t-1}{2}\Big)k+(k-1)+k\geq \Big(\frac{n-t-1}{2}\Big)k+l-1.$$
		Depth Lemma implies,
		$$\depth(S/I^t) \geq  \Big(\frac{n-t-1}{2}\Big)k+l-1.$$
This completes the proof for depth. Note that from Lemma \ref{le7} and Theorem \ref{le12} we have that $\sdepth (S/I^t)\geq 1$, for all $t\geq 1$. Proof for the Stanley depth is similar by using Lemma \ref{le9} instead of Depth Lemma.
\end{proof}
\end{Theorem}
\begin{Corollary}\label{cor3}
	If $n\geq 2$, $t\geq 1$ and $I=I(P_{n,k})$ then
	$$\depth(S/I^t),\sdepth(S/I^t) \geq \left\{\begin{array}{ll}	\max\{1,\Big(\frac{n-t+1}{2}\Big)k-1\}\,, & \text{if\,\, $n$ and $t$ have opposite parity;} \\
	\max\{1,\Big(\frac{n-t}{2}\Big)k\} \,, & \text{if\,\, $n$ and $t$ have same parity.}
	
	\end{array}\right.$$
\end{Corollary}
\noindent A comparison of the actual values of depth with lower bound in Corollary \ref{cor3} is shown in the following example.
\begin{Example} {\em
By using CoCoA we have, $\depth(S/I^2(P_{4,4}))=5$ and $\depth(S/I^2(P_{5,3}))=6$, while by our Corollary \ref{cor3}, $\depth(S/I^2(P_{4,4}))\geq 4$ and $\depth(S/I^2(P_{5,3}))\geq 5$.
}

\end{Example}
\noindent Also this new bound is much sharper than the one given in Corollary \ref{dsc}, as shown in the following example.
\begin{Example}{\em
Let $I=I(P_{n,k})$ with $n=50$ and $k=10$. Clearly $P_{n,k}$ has two near leaves and its diameter is $51$. Let $t=15$. By Corollary \ref{dsc}
	$$\depth(S/I^{15})\,,\,\sdepth(S/I^{15}) \geq  \big \lceil \frac{51-15+2}{3} \big \rceil  = 13.$$ Whereas our Corollary \ref{cor3} shows that 	 $$\depth(S/I^{15})\,,\,\sdepth(S/I^{15}) \geq \Big (\frac{50-15+1}{2}\Big )10-1 = 179.$$
Comparison shows a noteworthy difference between both the lower bounds.}
\end{Example}

\section{Powers of edge ideal of a Subclass of Lobster Tree}
Let $r\geq 2$ and $p,t\geq 1$ be some integers. In this section we give an upper bound for depth and Stanley depth of $S/I^t(\mathcal{S}_{r,p})$. Our bounds depends only on $r$ and $t$. We significantly improve the bound for the depth and Stanley depth of $S/I^t(\mathcal{S}_{r,p})$ given in Corollary \ref{dsc}. It is easy to see that the diameter of $\mathcal{S}_{r,p}$ is fixed for any $r$ and $p$. The bound given in Corollary \ref{dsc} depends on $t$ and diameter of $\mathcal{S}_{r,p}$ so this bound becomes weak for bigger values of $t$. Where as our bound given in Corollary \ref{cor4} being independent of the diameter of $\mathcal{S}_{r,p}$ is better. Before proving the results of this section we introduce some notations. Let $p\geq 1$ and $q\geq 0$ be integers such that $q\leq p$. Let $1\leq i\leq r-1$, $B_i:=\{x_{1i},x_{2i},\dots,x_{pi}\}$, $B_r:=\{x_{1r},x_{2r},\dots,x_{qr}\}$ ($B_r=\emptyset$, if $q=0$), $\bar{B}_i:=B_i\cup \{v_i\}$ and $\bar{B}_{r}:=B_r\cup \{v_r\}$.  Let $B:=\{v_{r+1}\}\cup\bar{B}_1\cup \bar{B}_2\cup\dots\cup \bar{B}_r$ and define $S:=K[B].$

\begin{figure}[h!]
		\centering		\includegraphics[width=6cm]{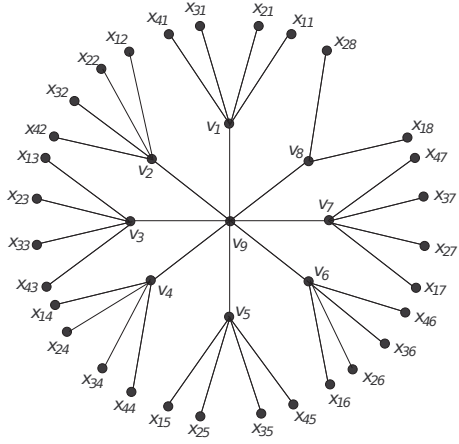}
		\caption{Graph $\mathcal{S}_{8,4,2}$ with labelled vertices.}
	\end{figure}

If $q\geq 1$, then
$$G(I(\mathcal{S}_{r,p,q}))=\{v_{r+1}v_r,v_rx_{1r},v_rx_{2r},\dots,v_rx_{qr}\}\cup\bigcup_{i=1}^{r-1}\{v_{r+1}v_i,v_ix_{1i},v_ix_{2i},\dots,v_ix_{pi}\}.$$
If $q=0$, then $$G(I(\mathcal{S}_{r,p,q}))=\{v_{r+1}v_r\}\cup\bigcup_{i=1}^{r-1}\{v_{r+1}v_i,v_ix_{1i},v_ix_{2i},\dots,v_ix_{pi}\}.$$
Note that for $1\leq j\leq r-1$ we have that $$I(\mathcal{S}_{j,k}):=I(\mathcal{S}_{r,p,q})\cap K[\{v_{r+1}\}\cup \bar{B}_1\cup  \dots\cup\bar{B}_{j-1} \cup \bar{B}_{j}].$$
Before proving the main result of this section we prove the result when $t=1$ in the following lemma.
\begin{Lemma} \label{lelt}
	Let $r\geq 2$ and $I=I(\mathcal{S}_{r,p,q})$. We have that $$\depth(S/I)\,,\, \sdepth(S/I) \geq \left\{
                                                                                          \begin{array}{ll}
                                                                                            r-1, & \hbox{if $q=0$;} \\
                                                                                            r, & \hbox{otherwise.}
                                                                                          \end{array}
                                                                                        \right.
.$$
\end{Lemma}
\begin{proof}
		Consider the short exact sequence
	$$0\longrightarrow\ S/(I:v_{r+1}) \xrightarrow{\,\cdot v_{r+1}\,}\ S/I
	\longrightarrow\ S/(I,v_{r+1}) \longrightarrow\ 0,$$
	we have $S/(I:v_{r+1})\cong K[B_1\cup B_2\cup \dots\cup B_r\cup \{v_{r+1}\}]$, therefore 	
		$$\depth(S/(I:v_{r+1}))= (r-1)p +q+1\geq r.$$
	By definition of $I$, $(I,v_{r+1})=(I(H),v_{r+1})$, where $H$ is a forest with $r$ connected components, say $H_1,H_2,\dots,H_{r-1},H_r$. It can easily be seen that among these $r$ connected components, $r-1$ components are $p$-star graphs while one component is a $q$-star. Without loss of generality we may assume that for $1\leq i\leq r-1$, $H_i\cong \mathcal{S}_p$ and $H_r\cong \mathcal{S}_q$. If $q=0$ then $H_r\cong \mathcal{S}_0$ is a trivial graph on one vertex, say $v$. From Lemma \ref{ddisjoint} and Proposition \ref{2} it follows that
\begin{multline*}
\depth(S/(I,v_{r+1}))=\depth(K[B\backslash \{v_{r+1}\}]/(I(H)))=\depth(K[V(H_1)]/I(H_1))+ \dots +\\\depth(K[V(H_{r-1})]/I(H_{r-1}))+\depth(K[v]/(v))=\underbrace{1+1+\ldots+1}_{(r-1)- \text{ times}}+0=r-1.
\end{multline*}
If $q\neq0$ then $H_r\cong \mathcal{S}_q$. From Lemma \ref{ddisjoint} and Proposition \ref{2} it follows that
\begin{multline*}
\depth(S/(I,v_{r+1}))=\depth(K[B\backslash \{v_{r+1}\}]/(I(H)))=\depth(K[V(H_1)]/I(H_1))+ \dots +\\\depth(K[V(H_{r-1})]/I(H_{r-1}))+\depth(K[V(H_r)]/I(H_r))=\underbrace{1+1+\ldots+1}_{r- \text{ times}}=r.
\end{multline*}
Hence by applying Depth Lemma the required result follows. The result for Stanley depth can be proved in the same lines by using Lemma \ref{le9} instead of Depth Lemma and Lemma \ref{sdisjoint} instead of Lemma \ref{ddisjoint}.
	\end{proof}
\begin{Corollary}\label{sqr}{\em
Let $r\geq 2$ and $I=I(\mathcal{S}_{r,p})$. We have that $$\depth(S/I)\,,\, \sdepth(S/I) \geq r.$$}
\end{Corollary}
\begin{Example}
{\em We use CoCoA and show that the equality may hold in Corollary \ref{sqr}. For instance, we have $\depth(S/I(\mathcal{S}_{4,2}))=\sdepth(S/I(\mathcal{S}_{4,2}))=4$ and $\depth(S/I(\mathcal{S}_{5,2}))$= $\sdepth(S/I(\mathcal{S}_{5,2}))=5$.}
\end{Example}

\begin{Lemma}\label{new}
If $I=I(\mathcal{S}_{r,p,q})$ then for $t\geq 1$, $(I^t,v_r)=(I^t(\mathcal{S}_{r-1,p}),v_r)$.
\end{Lemma}
\begin{proof}
The inclusion $(I^t(\mathcal{S}_{r-1,p}),v_r)\subseteq(I^t,v_r)$ is clear. Conversely, if $w\in I^t$ is a monomial which is not divisible by $v_r$, then, by definition of $G(I)$, it follows that $w\in I^t(\mathcal{S}_{r-1,p})$.
\end{proof}
\noindent Now moving towards the main result of this section.
\begin{Theorem} \label{thss}
Let $r\geq 2$, $t\geq 1$, $p\geq 1$ and $0 \leq q \leq p$. If $I=I(\mathcal{S}_{r,p,q})$ then
	$$
	\depth(S/I^t)\, , \, 	\sdepth(S/I^t) \geq  \left\{
	\begin{array}{ll}
\max\{1,r-t\}, & \qquad \text{if $q=0$;} \\
	\max\{1,r-t+1\}, & \qquad \text{otherwise.}
	\end{array}
	\right.
	$$
\end{Theorem}
\begin{proof}		
We use induction on $r$ and $t$. If $r=2$ and $ t \geq 1$, the result follows from Lemma \ref{le7}. If $t=1$ and $r \geq 2$, the result follows from Lemma \ref{lelt}. Assume $r \geq 3$ and $t \geq 2$.	 Consider the short exact sequence
	\begin{equation}
	0\longrightarrow\ S/(I^t:v_r) \xrightarrow{\,\cdot v_r\,}\ S/I^t
	\longrightarrow\ S/(I^t,v_r) \longrightarrow\ 0
	\end{equation}
 by Depth Lemma
	\begin{equation} \label{2t1}
	\depth(S/I^t)\geq \min\{\depth(S/(I^t:v_r)),
	\depth(S/(I^t,v_r))\}.
	\end{equation}
	\textbf{Case 1:} Let $q=0$. From Lemma \ref{new} it follows that
	$(I^t,v_r) = (I^t(\mathcal{S}_{r-1,p}),v_r)$. Since $\mathcal{S}_{r-1,p}=\mathcal{S}_{r-1,p,p}$ and $p\geq 1$, using induction on $r$, it follows that
	\begin{eqnarray*}
		\depth(S/(I^t,v_r)) &=& \depth (K[{B}\backslash \{v_r\}]/I^t(\mathcal{S}_{r-1,p})) \geq (r-1)-t+1=r-t.
	\end{eqnarray*}	
	We consider the short exact sequence
	\begin{equation} \label{2t2}
	0\longrightarrow\
	S/(I^t:v_rv_{r+1}) \xrightarrow{\,\cdot v_{r+1}\,}\ S/(I^t:v_r )\longrightarrow\
	S/((I^t:v_r),v_{r+1}) \longrightarrow\ 0,
	\end{equation}
 since by Lemma \ref{le2}, $(I^t:v_rv_{r+1}) = I^{t-1}$ so by induction on $t$
	\begin{eqnarray*}
		\depth(S/(I^t:v_rv_{r+1})) &=& \depth(S/I^{t-1})   \geq r-(t-1) =r-t+1.
	\end{eqnarray*}
 Let $R' =K[B\backslash \{v_{r+1}\}]$ and $I'=IR'$. By Lemma \ref{morey}, $S/((I^t:v_r),v_{r+1}) \cong R'/ (I^t:v_r)\cong R'/(I')^t$. Clearly $v_r$ is a regular variable on $R'/ (I')^t$ and $I'$ corresponds to the edge ideal of a forest consisting of $r-1$ connected components and each component is a $p$-star. Therefore by Lemma \ref{le15} and  Theorem \ref{rth1}
	\begin{eqnarray*}
	\depth(	S/((I^t:v_r),v_{r+1}))=\depth(R'/(I')^t)&\geq& \max \big \{\big \lceil \frac{2-t+2}{3} \big \rceil +(r-1)-1 , r-1\big \} +1  \\ &=&  (r-1) +1= r > r-t.
	\end{eqnarray*}	
By applying Depth Lemma on sequence (\ref{2t2}) we get
$\depth(S/ (I^t:v_r)) \geq r-t.$
From  Eq. (\ref{2t1}) the result follows.\\\\
\noindent \textbf{Case 2:} Let $q \geq 1$. Let us label the vertices of  $B_r\neq \emptyset$ with $\{y_1,y_2,\dots,y_q\}$. By Lemma \ref{new}, $(I^t,v_r)=
	(I^t(\mathcal{S}_{r-1,p}),v_r)$, therefore $S/(I^t,v_r)\cong (K[B\backslash\bar{B}_r]/I^t(\mathcal{S}_{r-1,p}))[B_r]$. Thus by Lemma \ref{le15} and induction on $r$
	\begin{eqnarray*}
		\depth(S/(I^t,v_r)) &=& \depth (S/I^t (\mathcal{S}_{r-1,p}))+|B_r| \\
		&\geq& ((r-1)-t+1)+q
	= q+r-t,
	\end{eqnarray*}
	Let $R_i = S/(y_1 , \dots , y_i)$ and $I_i=IR_i$, where $R_0=S$ and $I_0=I$. We consider a family of short exact sequences:
	$$0\longrightarrow\
	R_0/(I_0^t:v_ry_1) \xrightarrow{\,\cdot y_1\,}\ R_0/(I^t:v_r )\longrightarrow\
	R_0/((I^t:v_r),y_1) \longrightarrow\ 0$$
	$$0\longrightarrow\
	R_1/(I_1^t:v_ry_2) \xrightarrow{\,\cdot y_2\,}\ R_1/(I_1^t:v_r )\longrightarrow\
	R_1/((I_1^t:v_r),y_2) \longrightarrow\ 0$$
	$$0\longrightarrow\
	R_2/(I_2^t:v_ry_3) \xrightarrow{\,\cdot y_3\,}\ R_2/(I_2^t:v_r )\longrightarrow\
	R_2/((I_2^t:v_r),y_3) \longrightarrow\ 0$$
	$$\vdots$$
	$$0\longrightarrow\
	R_{q-1}/(I_{q-1}^t:v_ry_{q}) \xrightarrow{\,\cdot y_q\,}\ R_{q-1}/(I_{q-1}^t:v_r )\longrightarrow\
	R_{q-1}/((I_{q-1}^t:v_r),y_{q}) \longrightarrow\ 0.$$
	For $0 \leq i \leq q-1$, by Lemma \ref{le2} we have $R_i/(I_i^t : v_ry_{i+1}) \cong R_i /I_i ^{t-1}$. Thus by induction on $t$ \begin{equation}
	\depth(R_i/(I_i^t : v_ry_{i+1})) = \depth(R_i /I_i ^{t-1}) \geq r-(t-1)+1 = r-t+2.
	\end{equation}	
By Lemma \ref{morey}, $R_{q-1}/((I_{q-1}^t:v_r),y_{q})\cong R_{q}/(I_{q}^t:v_r)$, now we have the short exact sequence
$$0\longrightarrow\ R_{q}/(I_{q}^t:v_rv_{r+1}) \xrightarrow{\,\cdot v_{r+1}\,}\ R_{q}/(I_{q}^t:v_r )\longrightarrow\
	R_{q}/((I_{q}^t:v_r),v_{r+1}) \longrightarrow\ 0,$$
by Lemma \ref{le2} we have $\depth(R_{q}/(I_{q}^t:v_rv_{r+1}))=\depth(R_{q}/I_{q}^{t-1}).$ Thus it is easy to see that $R_{q}/I_{q}\cong K[B\backslash B_r]/I(\mathcal{S}_{r,p,0})$. Thus by induction on $t$ and case (1), $\depth(R_{q}/I_{q}^{t-1})\geq r-(t-1)=r-t+1.$
Clearly $R_{q}/((I_{q}^t:v_r),v_{r+1})\cong R''/L^t$, where $R''=[B\backslash (B_r\cup\{v_{r+1}\})]$ and $L=IR''$ is the edge ideal of a forest consisting of $r-1$ connected components and each component is a $p$-star. Clearly $v_r$ is a regular variable on $R''/L^t$. Therefore by Lemma \ref{le15} and  Theorem \ref{rth1}
\begin{eqnarray*}
	\depth(R_{q}/((I_{q}^t:v_r),v_{r+1}))=\depth(R''/L^t)&\geq& \max \big \{\big \lceil \frac{2-t+2}{3} \big \rceil +(r-1)-1 , r-1 \big \} +1  \\ &=&  (r-1) +1= r \geq r-t+1.
\end{eqnarray*}		
\noindent Thus, by Depth Lemma $\depth(S/(I^t:v_r)) \geq r-t+1,$ and hence by Eq. (\ref{2t1}) $\depth(S/I^t) \geq  r-t+1. $
On the same lines by using Lemma \ref{le9} instead of Depth Lemma one can prove the result for Stanley depth.
\end{proof}
\begin{Corollary}\label{cor4}
Let $t\geq1$, $p\geq 1$ and $I=I(\mathcal{S}_{r,p})$. We have that
	$$\depth(S/I^t)\, , \, 	\sdepth(S/I^t)\geq \max\{1, r-t+1\}.$$
\end{Corollary}
\noindent A comparison of the actual values of depth with lower bound in Corollary \ref{cor4} is shown in the following example.
\begin{Example}
{\em By using CoCoA we have, $\depth(S/I^2(\mathcal{S}_{4,2}))=4$ and $\depth(S/I^2(\mathcal{S}_{5,2}))=5$, while by our Corollary \ref{cor4}, $\depth(S/I^2(\mathcal{S}_{4,2}))\geq 3$ and $\depth(S/I^2(\mathcal{S}_{5,2}))\geq 4$. }
\end{Example}
\noindent Also this new bound is much sharper than the one given in Corollary \ref{dsc}, as shown in the following example. Note that $\mathcal{S}_{r,p}$ has
$r$ near leaves.
\begin{Example} {\em
Let $I=I(\mathcal{S}_{r,p})$ with $r =55$ and $t =10$. Clearly $d=4$, thus by Corollary \ref{dsc} we have $$\depth(S/I^{10})\,,\,\sdepth(S/I^{10}) \geq \big \lceil \frac{4-10+55}{3} \big \rceil  = 17.$$ While by our Corollary \ref{cor4},  $$\depth(S/I^{10})\,,\,\sdepth(S/I^{10}) \geq 55-10+1 = 46.$$
\\}
\end{Example}
	

\begin{thebibliography}{99}
	\bibitem {AA} Alipour, A., Tehranian, A. (2017). Depth and Stanley Depth of Edge Ideals of Star Graphs.{\emph{ International Journal of Applied Mathematics and Statistics}}, 56(4), 63-69.
	

\bibitem{Br} Brodmann, M. (1979). The asymptotic nature of the analytic spread. \emph{Mathematical Proceedings of the Cambridge Philosophical Society}, 86(1), 35-39.	
\bibitem{B} Burch, L. (1972). Codimension and analytic spread. \emph{Mathematical Proceedings of the Cambridge Philosophical Society}, 72(3), 369-373.

\bibitem{BH} Bruns, W., Herzog, H. J.(1998). Cohen-Macaulay rings. {Cambridge University Press}.	
\bibitem{FF} Cimpoeas, M. (2008). Some remarks on the Stanley's depth for multigraded modules. \emph{
Le Matematiche}, Vol. LXIII - Fasc. II, 165-171.
\bibitem{COC} CoCoATeam, CoCoA: A system for doing Computations in Commutative Algebra, available at http://cocoa.dima.unige.it.
	\bibitem{D} 
	Duval, A. M., Goeckner, B., Klivans, C. J.,  Martin, J. L. (2016). A non-partitionable Cohenâ-Macaulay simplicial complex.{\emph{ Advances in Mathematics}}, 299, 381-395.
	
	 \bibitem {LS} 
	Fouli, L.,  Morey, S. (2015). A lower bound for depths of powers of edge ideals.{\emph{ Journal of Algebraic Combinatorics}}, 42(3), 829-848.
	
		\bibitem{JA} 
	Gallian, J. A. (2009). A dynamic survey of graph labeling. {\emph{The Electronic Journal of Combinatorics}}, 16(6), 1-219.
	
	\bibitem {HH} Herzog, J. ,Hibi, T. (2005). The depth of powers of an ideal.  {\emph {Journal of Algebra}}, 291(2), 534-550.
	
	 \bibitem{HV} 
	Herzog, J., Vladoiu, M.,  Zheng, X. (2009). How to compute the Stanley depth of a monomial ideal. {\emph { Journal of Algebra}}, 322(9), 3151-3169.
	
	\bibitem{MI}
	Ishaq, M. (2011). Values and bounds for the Stanley depth.{\emph { Carpathian Journal of Mathematics}}, 27(2), 217-224.
	
	\bibitem{IQ}
	Ishaq, M., Qureshi, M. I. (2013). Upper and lower bounds for the Stanley depth of certain classes of monomial ideals and their residue class rings.{\emph {  Communications in Algebra}}, 41(3), 1107-1116.
	

	\bibitem {SM}
	Morey, S. (2010). Depths of powers of the edge ideal of a tree.{\emph{ Communications in Algebra}}, 38(11), 4042-4055.
	
	\bibitem{pop}	
	Popescu, A. (2010). Special stanley decompositions. \emph{ Bulletin math$\acute{e}$matique de la Soci$\acute{e}$t$\acute{e}$ des Sciences Math$\acute{e}$matiques de Roumanie}, 53(101), No. 4, 363-372.
	
	 \bibitem{FA} 
	Pournaki, M., Seyed Fakhari, S. A.,  Yassemi, S. (2013). Stanley depth of powers of the edge ideal of a forest.{\emph{ Proceedings of the American Mathematical Society}}, 141(10), 3327-3336.
\bibitem{GR} Rinaldo, G. (2008). An algorithm to compute the Stanley depth of monomial ideals, \emph{Le Matematiche}, Vol. LXIII - Fasc. II, 243-256.	 
	
\bibitem{4} 
	Rauf, A. (2010). Depth and Stanley depth of multigraded modules.{\emph{ Communications in Algebra}}, 38(2), 773-784.
	
	\bibitem{ST} 
	Stanley, R. P. (1982). Linear Diophantine equations and local cohomology.{\emph{ Inventiones mathematicae}}, 68(2), 175-193.
	
		\bibitem{V} 
	Villarreal, R. H. (2001). Monomial Algebras. Monographs and Textbooks in Pure and Applied Mathematics. New York: Marcel Dekker, Inc., Vol. 238.
	
	
\end{thebibliography}
\end{document}